\documentclass{amsart}
\usepackage{amsmath,amssymb,amsthm}
\usepackage{enumitem}

\usepackage[colorlinks=true,linkcolor=blue,citecolor=blue,pdfpagelabels=false]{hyperref}

\theoremstyle{plain}
  \newtheorem{theorem}{Theorem}[section]
  \newtheorem{lemma}[theorem]{Lemma}
  \newtheorem{conjecture}[theorem]{Conjecture}
  \newtheorem{corollary}[theorem]{Corollary}
  \newtheorem{proposition}[theorem]{Proposition}
  
\theoremstyle{definition}
  \newtheorem{example}[theorem]{Example}
  \newtheorem{remark}[theorem]{Remark}
  \newtheorem{question}[theorem]{Question}

\newenvironment{acknowledgements}{\bigskip\textbf{Acknowledgements.}}{}
\newenvironment{enumeratealpha}{\begin{enumerate}[label=(\alph*)]}{\end{enumerate}}

\newcommand{\dueto}[1]{\textup{\textbf{(#1) }}}
\newcommand{\equallim}{\mathop{=}\limits}

\renewcommand{\geq}{\geqslant}
\renewcommand{\leq}{\leqslant}
\newcommand{\op}[1]{\ensuremath{\operatorname{#1}}}

\renewcommand{\Im}{\op{Im}\,}

\title{On multiple and infinite log-concavity}

\author{Luis A. Medina}
\address{Department of Mathematics,
University of Puerto Rico,
Box 70377,
San Juan, PR 00936-8377,
Puerto Rico}
\email{luis.medina17@upr.edu}

\author{Armin Straub}
\address{Department of Mathematics,
University of Illinois at Urbana-Champaign,
1409 W. Green St,
Urbana, IL 61801,
United States}
\curraddr{Max-Planck-Institut f\"ur Mathematik,
Vivatsgasse 7,
53111 Bonn,
Germany}
\email{astraub@illinois.edu}

\begin{document}

\begin{abstract}
  Following Boros--Moll, a sequence $(a_n)$ is $m$-log-concave if
  $\mathcal{L}^j (a_n) \geq 0$ for all $j = 0, 1, \ldots, m$. Here,
  $\mathcal{L}$ is the operator defined by $\mathcal{L} (a_n) = a_n^2 - a_{n -
  1} a_{n + 1}$. By a criterion of Craven--Csordas and McNamara--Sagan it is
  known that a sequence is $\infty$-log-concave if it satisfies the stronger
  inequality $a_k^2 \geq r a_{k - 1} a_{k + 1}$ for large enough $r$. On
  the other hand, a recent result of Br\"and\'en shows that
  $\infty$-log-concave sequences include sequences whose generating polynomial
  has only negative real roots. In this paper, we investigate sequences which
  are fixed by a power of the operator $\mathcal{L}$ and are therefore
  $\infty$-log-concave for a very different reason. Surprisingly, we find that
  sequences fixed by the non-linear operators $\mathcal{L}$ and
  $\mathcal{L}^2$ are, in fact, characterized by a linear 4-term recurrence.
  In a final conjectural part, we observe that positive sequences appear to
  become $\infty$-log-concave if convoluted with themselves a finite number of
  times.
\end{abstract}

% 05A20 Combinatorics: Combinatorial inequalities
% 39B12 Functional equations and inequalities: Iteration theory, iterative and composite equations


% \begin{keyword}
%   log-concavity \sep linear recurrences \sep convolution
%   \MSC[2010] 05A20 \sep 39B12
% \end{keyword}

\date{May 7, 2014}
% \date{\today}

\maketitle

\section{Introduction}

A sequence $(a_n)$ is said to be log-concave if $a_n^2 \geq a_{n - 1}
a_{n + 1}$ for all $n$. If all terms of the sequence are positive, then log-concavity
implies unimodality. For a very nice introduction and many examples of both
unimodal and log-concave sequences we refer to \cite{stanley1989a}.

Following Boros and Moll \cite{moll-irresistible}, we denote with
$\mathcal{L}$ the operator which sends a sequence $(a_n)$ to the sequence
$(a_n^2 - a_{n - 1} a_{n + 1})$. Then $(a_n)$ is log-concave if and only if
$\mathcal{L} (a_n) \geq 0$. Similarly, the sequence $(a_n)$ is said to be
$m$-log-concave if $\mathcal{L}^j (a_n) \geq 0$ for all $j = 0, 1,
\ldots, m$. If $(a_n)$ is $m$-log-concave for all $m > 0$, then it is said to
be $\infty$-log-concave. Often, we will consider the generating function $f
(x) = \sum_{n \geq 0} a_n x^n$. In that case, we write $\mathcal{L} [f]
(x) = \sum_{n \geq 0} (a_n^2 - a_{n - 1} a_{n + 1}) x^n$ with the
understanding that $a_{- 1} = 0$.

\begin{example}
  \label{eg:binomfix}The sequence $\binom{n}{2}$ is fixed by the operator
  $\mathcal{L}$ because
  \begin{equation*}
    \binom{n}{2}^2 - \binom{n - 1}{2} \binom{n + 1}{2} = \binom{n}{2} .
  \end{equation*}
  Since the sequence is nonnegative, it is therefore $\infty$-log-concave. In
  Section~\ref{sec:fixedbyL}, we will characterize all sequences fixed by
  $\mathcal{L}$.
\end{example}

More generally, as a warm-up problem, it was conjectured in
\cite{moll-irresistible} that the binomial coefficients (either rows, that
is $\binom{n}{0}, \binom{n}{1}, \ldots, \binom{n}{n}$, or columns, that is
$\binom{n}{n}, \binom{n + 1}{n}, \binom{n + 2}{n}, \ldots$, of Pascal's
triangle) are $\infty$-log-concave. That the rows of Pascal's triangle are
indeed $\infty$-log-concave was recently proven by Br{\"a}nd{\'e}n
\cite{branden-inflog} who established, much more generally, that the
coefficients of polynomials, all of whose roots are negative and real, are
always $\infty$-log-concave. This connection of the location of roots and
log-concavity will be briefly reviewed in Section~\ref{sec:review}. We remark
that, on the other hand, the case of columns of Pascal's triangle is still
wide open; based on extensive computations by Kauers and Paule
\cite{kauers-paule} it is only known that they are $5$-log-concave.

As will also be reviewed in Section~\ref{sec:review}, a sequence is
$\infty$-log-concave if it satisfies the stronger condition $a_n^2 \geq r
a_{n - 1} a_{n + 1}$ for $r \geq (3 + \sqrt{5}) / 2 \approx 2.618$. This
criterion, due to Craven and Csordas \cite{cc-lc02} as well as McNamara and
Sagan \cite{mcnamara-sagan-inflog}, generalizes to a powerful approach of
showing that a specific sequence is $\infty$-log-concave. On the other hand,
there are sequences, like the one in Example~\ref{eg:binomfix}, that are fixed
by the operator $\mathcal{L}$ (or one of its powers) and are
$\infty$-log-concave for this reason. This leads us to investigate the
sequences fixed by $\mathcal{L}$ in Section~\ref{sec:fixedbyL} as well as the
sequences fixed by $\mathcal{L}^2$ in Section~\ref{sec:fixedbyL2}. In both
cases, we find that, surprisingly, these sequences fixed by non-linear
operators are characterized by {\emph{linear}} 4-term recurrences with
constant coefficients. This phenomenon does not appear to extend to the
sequences fixed by $\mathcal{L}^m$ for $m > 2$. As an application, we ask if
these criteria for $\infty$-log-concavity can be combined to yield an
algorithm that decides, in finite time, whether or not a given finite sequence
is $\infty$-log-concave.

In the final section, Section~\ref{sec:powers}, we start with positive
sequences and repeatedly convolute them with themselves. That is to say, we
consider the coefficients of powers of a given polynomial. The central limit
theorem suggests that, as the exponent increases, the shape of the resulting
sequences approaches the shape of a normal distribution. One therefore expects
the sequences to become more and more log-concave. Indeed, we observe that
each sequence appears to become $\infty$-log-concave in a finite number of
steps.

\section{Review of multiple log-concavity}\label{sec:review}

Newton's famous and classical theorem on real roots states that if the
polynomial $p (x) = a_0 + a_1 x + \ldots + a_d x^d$ has only negative real
roots, then its coefficients $(a_n)$ are log-concave. We refer to
\cite{niculescu-newton} for historic information and related results.

\begin{remark}
  In fact, somewhat stronger, Newton's theorem implies that the numbers $a_n /
  \binom{d}{n}$ form a log-concave sequence. This is indeed stronger since the
  binomial coefficients are log-concave and the Hadamard product of two
  log-concave sequences is again log-concave. If only log-concavity is
  desired, the assumptions of Newton's theorem can be weakened: it is shown in
  \cite{bahls-lc11} that the condition that all roots are negative (and
  therefore real) can be replaced with the condition that all roots lie in
  the sector defined by $| \arg (- z) | \leq \pi / 3$ (equivalently, all
  roots $z = a + b i$ satisfy both $a \leq 0$ and $b^2 \leq 3 a^2$).
\end{remark}

Br\"and\'en \cite{branden-inflog} recently established the following more
general theorem which was previously, and independently, conjectured by
Stanley, McNamara--Sagan \cite{mcnamara-sagan-inflog} and Fisk
\cite{fisk-q08}.

\begin{theorem}
  {\dueto{\cite{branden-inflog}}}Let $p (x) = \sum_{n = 0}^d a_n x^n$. If
  all roots of $p$ are negative and real then so are the roots of
  \begin{equation*}
    \mathcal{L} [p] (x) = \sum_{n = 0}^d (a_n^2 - a_{n - 1} a_{n + 1}) x^n .
  \end{equation*}
  Here, it is understood that $a_{- 1} = 0$ as well as $a_{d + 1} = 0$.
\end{theorem}

\begin{corollary}
  \label{cor:realrootslc}If the polynomial $p (x) = a_0 + a_1 x + \ldots + a_d
  x^d$ has only negative real roots, then its coefficients $(a_n)$ are
  $\infty$-log-concave.
\end{corollary}

\begin{example}
  If $a_n = \binom{d}{n}$, then $p (x) = \sum_{n = 0}^d a_n x^n = (x + 1)^d$
  clearly has only negative real roots. It follows that (rows of) the binomial
  coefficients are $\infty$-log-concave, confirming the motivating conjecture
  in \cite{moll-irresistible}.
\end{example}

Note that the converse of Corollary~\ref{cor:realrootslc} is not true, as is
illustrated by the polynomial
\begin{equation*}
  P_2 (x) = \frac{3}{2} x^2 + \frac{15}{4} x + \frac{21}{8},
\end{equation*}
which has non-real roots but coefficients which are $\infty$-log-concave; this
is easily proved using the notion of $r$-factor log-concavity which is
reviewed next. We remark that $P_2 (x)$ is one of the Boros--Moll polynomials,
which occurred in the evaluation of a quartic integral
\cite{moll-irresistible}, and have been investigated by many authors hence;
we only refer to the recent article \cite{cyz-borosmoll} as well as the references
therein. Nevertheless, as illustrated by Theorem~\ref{thm:kurtz} below, a
partial converse to Br\"and\'en's result is possible if $\infty$-log-concavity
is replaced by an even stronger property.

An important tool for establishing infinite log-concavity of a specific
sequence is the notion of $r$-factor log-concavity, which has been introduced
and used by McNamara and Sagan \cite{mcnamara-sagan-inflog}. Let $r
\geq 1$. A sequence $ (a_n)$ is $r$-factor log-concave if
\begin{equation*}
  a_n^2 \geq r a_{n - 1} a_{n + 1} .
\end{equation*}
In fact, using slightly different terminology, this notion has already been
considered by Craven and Csordas in \cite{cc-lc02}. Let $\mathcal{M}_r$ be
the $r$-factor log-concave sequences. It is shown in \cite[Theorem
4.1]{cc-lc02} that
\begin{equation}
  \mathcal{L} (\mathcal{M}_r) \subset \mathcal{M}_{r + s} \label{eq:rfactor}
\end{equation}
if and only if $r \geq \frac{3 + \sqrt{5 + 4 s}}{2}$. Generalizations,
for instance to the case of decreasing sequences, appear in
\cite{grabarek-phd}. Of particular importance for our purposes is the
following consequence of the containment \eqref{eq:rfactor}.

\begin{lemma}
  \label{lem:rfactor}If a sequence $ (a_n)$ is $r$-factor log-concave for some
  $r \geq \frac{3 + \sqrt{5}}{2} \approx 2.618$, then $\mathcal{L} (a_n)$
  is $r$-factor log-concave as well. In particular, the sequence $(a_n)$ is
  $\infty$-log-concave.
\end{lemma}

As illustrated in the next example, this gives rise to a more general
approach, used in \cite{mcnamara-sagan-inflog}, to show that a specific
sequence is $\infty$-log-concave.

\begin{example}
  \label{eg:rfactorinfty}The polynomial $p (x) = 1 + 4 x + 6 x^2 + 4 x^3$ has
  coefficients which are $2.25$-factor log-concave. The polynomial
  $\mathcal{L} [p] (x) = 1 + 10 x + 20 x^2 + 16 x^3$ has coefficients which
  are $2.5$-factor log-concave, while the coefficients of $\mathcal{L}^2 [p]
  (x) = 1 + 80 x + 240 x^2 + 256 x^3$ are $2.8125$-factor log-concave.
  Lemma~\ref{lem:rfactor} therefore shows that, in fact, the coefficients of
  $p (x)$ are $\infty$-log-concave.
\end{example}

The concept of $r$-factor log-concavity of the coefficients of a polynomial is
intimately related to the location of zeros discussed earlier in this section.
For instance, the following result on Hurwitz stability is due to Katkova and
Vishnyakova. We refer to \cite{katkova-stable} for further details,
including the cases of lower degree, and references to related and earlier
results. Here, {\emph{strong}} log-concavity, and likewise $r$-factor strong
log-concavity, mean that the defining inequalities are strict (for instance, a
sequence $ (a_n)$ is $r$-factor strongly log-concave if $a_n^2 > r a_{n - 1}
a_{n + 1}$).

\begin{theorem}
  {\dueto{\cite{katkova-stable}}}Let $p (x)$ be a polynomial with positive
  coefficients and degree larger than $5$. If the coefficients of $p (x)$ are
  $r_0$-factor strongly log-concave, where $r_0 \approx 1.466$ is the unique
  real root of $r^3 - r^2 - 1$, then all the roots of $p (x)$ have negative
  real parts.
\end{theorem}

As a second example, we cite the following result of Kurtz
\cite{kurtz-realroots}, as stated in \cite{csordas-mlc}, which may be
viewed as a partial converse to Newton's theorem on real roots.

\begin{theorem}
  {\dueto{\cite{kurtz-realroots}}}\label{thm:kurtz}Let $p (x)$ be a
  polynomial with positive coefficients. If the coefficients of $p (x)$ are
  $4$-factor log-concave, then all the roots of $p (x)$ are real (and hence
  negative).
\end{theorem}

It is also shown in \cite{kurtz-realroots} that $4$-factor log-concavity
cannot be replaced $(4 - \varepsilon)$-factor log-concavity for any
$\varepsilon > 0$.

\section{Sequences fixed by $\mathcal{L}$}\label{sec:fixedbyL}

As observed in Example~\ref{eg:binomfix}, the sequence $\binom{n}{2}$ is fixed
by $\mathcal{L}$ and therefore $\infty$-log-concave. In this section, we
characterize all sequences $(a_n)_{n \geq 0}$ that are fixed by the
operator $\mathcal{L}$, that is $\mathcal{L} (a_n) = (a_n)$, or, equivalently,
\begin{equation}
  a_n^2 - a_{n - 1} a_{n + 1} = a_n \label{eq:recLfix}
\end{equation}
for all indices $n \geq 0$ (with the understanding that $a_{- 1} = 0$).

\begin{remark}
  Let us note that the, apparently more general, characterization of sequences
  $(a_n)_{n \geq 0}$ such that $\mathcal{L} (a_n) = \lambda (a_n)$ for
  some number $\lambda$ reduces to the case $\lambda = 1$. Indeed, if
  $\mathcal{L} (a_n) = \lambda (a_n)$ then $\mathcal{L} (b_n) = (b_n)$ with
  $b_n = a_n / \lambda$.
\end{remark}

Suppose that $(a_n)_{n \geq 0}$ is a sequence fixed by $\mathcal{L}$.
Note that if $a_{m - 1} \neq 0$ and $a_m = 0$ for some $m > 0$, then $a_{m +
1} = 0$ as well. In particular, the sequence $(a_{n + m + 2})_{n \geq 0}$
is again fixed by $\mathcal{L}$. In characterizing all sequences that are
fixed by $\mathcal{L}$, it is therefore no loss of generality to assume that
the sequence $(a_n)$ has no {\emph{internal zeros}}, meaning that if $a_m = 0$
for some $m \geq 0$ then $a_n = 0$ for all $n \geq m$. If $(a_n)$
has no internal zeros, then $a_0 = 1$ unless $(a_n)$ is the zero sequence.
Assuming that $(a_n)$ is not the zero sequence, the value of $a_1$ then
determines $(a_n)$.

\begin{example}
  \label{eg:fixed}With the value $k$ given, let $(a_n)_{n \geq 0}$ be the
  unique sequence with $a_1 = k$ which is fixed by $\mathcal{L}$ and has no
  internal zeros. When $k \in \{0, 1, 2\}$ this sequence is finitely supported
  with values $(1)$, $(1, 1)$ and $(1, 2, 2, 1)$, respectively. On the other
  hand, if $k \geq 3$ is an integer, then the corresponding sequence is
  infinite and all of its terms are positive integers. These claims are not
  \textit{a priori} obvious but will be direct consequences of the
  characterization in Theorem~\ref{thm:fixL}. The first few cases, with $k =
  3, 4, 5$, are:
  \begin{eqnarray*}
    &  & 1, 3, 6, 10, 15, 21, 28, 36, \ldots\\
    &  & 1, 4, 12, 33, 88, 232, 609, 1596, \ldots\\
    &  & 1, 5, 20, 76, 285, 1065, 3976, 14840, \ldots
  \end{eqnarray*}
  In the first case, corresponding to $k = 3$, the sequence is given by $a_n =
  \binom{n + 2}{2}$, the introductory Example~\ref{eg:binomfix}, and in the
  case $k = 4$ we identify the sequence as $a_n = F_{2 n + 3} - 1$ with $F_n$
  denoting the Fibonacci numbers.
\end{example}

The alert reader may have noticed that the infinite sequences given in
Example~\ref{eg:fixed} each have a rational generating function. The next
result and its corollary show that this is always the case.

\begin{theorem}
  \label{thm:fixL}Let $k$ be an arbitrary number. Then the sequence $(a_n)_{n
  \geq 0}$ defined by
  \begin{equation*}
    \sum_{n \geq 0} a_n x^n = \frac{1}{1 - k x + k x^2 - x^3} =
     \frac{1}{(1 - x) (1 - (k - 1) x + x^2)}
  \end{equation*}
  is fixed by $\mathcal{L}$. Note that $a_0 = 1$, $a_1 = k$.
\end{theorem}

\begin{proof}
  Let $S$ be the (inverse) shift operator defined by $S a_n = a_{n - 1}$, and
  consider the operator
  \begin{equation}
    L := 1 - k S + k S^2 - S^3 = \left( 1 - S \right) \left( 1 - \left( k
    - 1 \right) S + S^2 \right) . \label{eq:recL}
  \end{equation}
  We note that the rational generating function for $\left( a_n \right)$ is
  equivalent to the fact that the recurrence $L a_n = 0$ holds, for all $n
  \geq 1$, with initial conditions $a_0 = 1$ and $a_{- 1} = a_{- 2} = 0$.
  
  The factorization of $L$, as on the right-hand side of \eqref{eq:recL},
  implies that the sequence $b_n = \left( 1 - \left( k - 1 \right) S + S^2
  \right) a_n = a_n - \left( k - 1 \right) a_{n - 1} + a_{n - 2}$ is constant.
  Since $b_0 = a_0 = 1$ it follows that $b_n = 1$ for all $n \geq 0$. In
  other words, $\left( a_n \right)$ satisfies the nonhomogeneous recurrence
  \begin{equation}
    a_n - \left( k - 1 \right) a_{n - 1} + a_{n - 2} = 1 \label{eq:recLnh}
  \end{equation}
  for all $n \geq 0$.
  
  To prove that $\left( a_n \right)$ is fixed under $\mathcal{L}$, we need to
  show that
  \begin{equation}
    a_n \left( a_n - 1 \right) = a_{n + 1} a_{n - 1} \label{eq:recFix}
  \end{equation}
  for all $n \geq 0$. Equation~\eqref{eq:recFix} clearly holds for $n =
  0$. For the purpose of induction, assume that \eqref{eq:recFix} holds for
  some $n$. Then
  \begin{eqnarray*}
    a_{n + 1} \left( a_{n + 1} - 1 \right) & \equallim^{( \ref{eq:recLnh})} &
    a_{n + 1} \left( \left( k - 1 \right) a_n - a_{n - 1} \right)\\
    & = & \left( k - 1 \right) a_{n + 1} a_n - a_{n + 1} a_{n - 1}\\
    & \equallim^{( \ref{eq:recFix})} & \left( k - 1 \right) a_{n + 1} a_n -
    a_n \left( a_n - 1 \right)\\
    & = & a_n \left( \left( k - 1 \right) a_{n + 1} - \left( a_n - 1 \right)
    \right)\\
    & \equallim^{( \ref{eq:recLnh})} & a_{n + 2} a_n,
  \end{eqnarray*}
  which shows that \eqref{eq:recFix} also holds for $n + 1$ in place of $n$.
  The result therefore follows by induction.
\end{proof}

\begin{remark}
  \label{rk:cfinite}Note that the sequence $(a_n)$, defined in
  Theorem~\ref{thm:fixL}, is $C$-finite \cite{zeilberger90}. Moreover, since
  $C$-finite sequences form an algebra, the sequences $(a_n^2)$ and $(a_{n -
  1} a_{n + 1})$, as well as any linear combination of these, are again
  $C$-finite. In order to obtain an alternative and automatic proof of
  Theorem~\ref{thm:fixL}, we can therefore use the $C$-finite ansatz, recently
  advertised in \cite{zeilberger-cfinite}, to show that
  \begin{equation*}
    a_n^2 - a_{n - 1} a_{n + 1} = a_n,
  \end{equation*}
  thus proving that $(a_n)$ is indeed fixed by $\mathcal{L}$.
\end{remark}

Theorem~\ref{thm:fixL} has the, possibly surprising, consequence that the
solutions of the non-linear three-term recurrence \eqref{eq:recLfix} in fact
satisfy a linear four-term recurrence.

\begin{corollary}
  \label{cor:rec}Let $N \in \mathbb{N} \cup \{\infty\}$. Suppose that the
  sequence $(a_n)_{n = 0}^N$ is fixed by $\mathcal{L}$ and that $a_n \neq 0$
  for all $0 \leq n \leq N$. Then, with $k = a_1$,
  \begin{equation}
    a_n - k (a_{n - 1} - a_{n - 2}) - a_{n - 3} = 0 \label{eq:recLfix4}
  \end{equation}
  for all $n = 1, 2, \ldots, N$, with initial conditions $a_0 = 1$ and $a_{-
  1} = a_{- 2} = 0$.
\end{corollary}

\begin{proof}
  Recall that if $(a_n)$ is fixed by $\mathcal{L}$ and $a_0 \neq 0$, then $a_0
  = 1$ and the value of $a_1$ determines the initial segment of nonzero terms
  of $(a_n)$. On the other hand, for any value $k = a_1$,
  Theorem~\ref{thm:fixL} provides a sequence that is fixed by $\mathcal{L}$.
  It follows that the two sequences have to agree for all initial nonzero
  terms.
\end{proof}

As illustrated by Example~\ref{eg:fixed}, the sequences fixed by $\mathcal{L}$
usually have infinite support. We now determine all sequences with finite
support that are fixed by $\mathcal{L}$.

\begin{proposition}
  \label{prop:psr:fix}Let $s \geq 3$ and $1 \leq r < s$ be integers.
  If $s \neq 2 r$, then
  \begin{equation}
    p_{s, r} (x) = \frac{1 - x^s}{(1 - x) (1 - 2 \cos (2 \pi r / s) x + x^2)}
    \label{eq:psr}
  \end{equation}
  is a degree $s - 3$ polynomial, and the coefficients of $p_{s, r}$ are fixed
  by $\mathcal{L}$. Moreover, every finitely supported sequence which is fixed
  by $\mathcal{L}$ and has no internal zeros arises in this way.
\end{proposition}

\begin{proof}
  Assume that $(a_n)$ is one of the sequences of Theorem~\ref{thm:fixL} with
  the property that there is $N \geq 0$ such that $a_{N + 1} = 0$. Choose
  the minimal such $N$. Then
  \begin{equation*}
    (a_n)_{n \geq 0} = (1, a_1, \ldots, a_{N - 1}, 1, 0, 0, 1, a_1,
     \ldots) .
  \end{equation*}
  Conversely, consider a finite sequence $(a_n)_{n = 0}^N$ of nonzero terms
  which is fixed by $\mathcal{L}$ (so that, in particular, $a_0 = a_N = 1$).
  Then $p (x) = a_0 + a_1 x + \ldots + a_N x^N$ necessarily satisfies
  \begin{equation*}
    \frac{p (x)}{1 - x^{N + 3}} = \frac{1}{(1 - x) (1 - (a_1 - 1) x + x^2)} .
  \end{equation*}
  It follows that $1 - (a_1 - 1) x + x^2 = (x - \zeta_1) (x - \zeta_2)$ where
  $\zeta_1, \zeta_2$ are nontrivial $s$-th roots of unity, with $s = N + 3$.
  In fact, one clearly has $\zeta_2 = \zeta_1^{- 1} = \overline{\zeta_1}$.
  Writing $\zeta_1 = \exp (2 \pi i r / s)$ for some $r$, we thus find
  \begin{equation*}
    a_1 - 1 = \zeta_1 + \overline{\zeta_1} = 2 \cos \left( \frac{2 \pi r}{s}
     \right) .
  \end{equation*}
  The claims follow from here.
\end{proof}

We observe that the polynomials $p_{s, r} (x)$ are palindromic, that is, their
coefficients form symmetric sequences. This property follows from $x^{s - 3}
p_{s, r} (1 / x) = p_{s, r} \left( x \right)$ which is readily verified from
\eqref{eq:psr}. While the coefficients of these polynomials are not in general
positive, we now show they are positive in the case $r = 1$. In this case, the
coefficients therefore form an interesting family of sequences which are
$\infty$-log-concave; see Example~\ref{eg:ps1} which lists the first few
sequences explicitly.

\begin{theorem}
  \label{thm:ps1}Let $s \geq 3$ be an integer. The polynomials
  \begin{equation}
    p_s (x) = p_{s, 1} (x) = \frac{1 - x^s}{(1 - x) (1 - 2 \cos (2 \pi / s) x
    + x^2)} \label{eq:ps1}
  \end{equation}
  have the following properties:
  \begin{enumeratealpha}
    \item \label{l:ps:palin}The polynomial $p_s (x)$ is palindromic of degree
    $s - 3$.
    
    \item \label{l:ps:pos}The coefficients of $p_s (x)$ are positive.
    
    \item \label{l:ps:fix}The coefficients of $p_s (x)$ are fixed by
    $\mathcal{L}$.
    
    \item \label{l:ps:inf}The coefficients of $p_s (x)$ are
    $\infty$-log-concave.
    
    \item \label{l:ps:r}The coefficients of $p_s (x)$ are $r$-factor
    log-concave if and only if
    \begin{equation*}
      r \leq \left\{ \begin{array}{ll}
         \frac{1}{\cos (2 \pi / s)}, & \text{if $s$ is even,}\\
         \frac{1}{[1 - 2 \cos (\pi / s)]^2}, & \text{if $s$ is odd.}
       \end{array} \right.
    \end{equation*}
  \end{enumeratealpha}
\end{theorem}

\begin{proof}
  Part \ref{l:ps:palin} is a special case of the palindromicity of the
  polynomials $p_{s, r} (x)$ which was observed above. Note that claim
  \ref{l:ps:fix} was proved in Proposition \ref{prop:psr:fix}. Together with
  the positivity claimed in part \ref{l:ps:pos}, this implies the
  $\infty$-log-concavity of part \ref{l:ps:inf}. It therefore only remains to
  show parts \ref{l:ps:pos} and \ref{l:ps:r}.
  
  Let us prove \ref{l:ps:pos}. Since the polynomial $p_s (x) = a_0 + a_1 x +
  \ldots + a_{s - 3} x^{s - 3}$ is palindromic of degree $s - 3$, it is
  sufficient to show that $a_j > 0$ for $j = 0, 1, \ldots, k$ where $k =
  \lfloor (s - 3) / 2 \rfloor$. We recall the classical generating function
  \begin{equation*}
    \frac{1}{1 - 2 z x + x^2} = \sum_{n = 0}^{\infty} U_n (z) x^n
  \end{equation*}
  of the Chebyshev polynomials $U_n (z)$ of the second kind. Since $U_n (\cos
  (\theta)) = \frac{\sin ((n + 1) \theta)}{\sin (\theta)}$, we have
  \begin{equation*}
    \frac{1}{1 - 2 \cos (\theta) x + x^2} = \frac{1}{\sin (\theta)} \sum_{n =
     0}^{\infty} \sin ((n + 1) \theta) x^n
  \end{equation*}
  and hence
  \begin{equation}
    p_s (x) = \frac{1 + x + \ldots + x^{s - 1}}{\sin (2 \pi / s)} \sum_{n =
    0}^{\infty} \sin (2 \pi (n + 1) / s) x^n . \label{eq:ps1sin}
  \end{equation}
  Note that $\sin (2 \pi / s) > 0$. The coefficient $a_j$ therefore is a
  positive linear combination of $\sin (2 \pi (n + 1) / s)$ for $n = 0, 1,
  \ldots, j$. On the other hand, for $n = 0, 1, \ldots, k$,
  \begin{equation*}
    \frac{2 \pi (n + 1)}{s} \leq \frac{2 \pi (k + 1)}{s} \leq
     \frac{\pi (s - 1)}{s} < \pi
  \end{equation*}
  and hence $\sin (2 \pi (n + 1) / s) > 0$. Therefore, the coefficients of
  $p_s (x)$ are positive as claimed in \ref{l:ps:pos}.
  
  To show part \ref{l:ps:r}, recall that $(a_n)$ is $r$-factor log-concave if
  and only if $a_n^2 - r a_{n + 1} a_{n - 1} \geq 0$ for all $n = 1, 2,
  \ldots, s - 4$. Using the fact that $a_n^2 - a_{n + 1} a_{n - 1} = a_n$,
  this inequality is clearly equivalent to
  \begin{equation*}
    r \leq \frac{a_n^2}{a_{n + 1} a_{n - 1}} = \frac{a_n}{a_n - 1} .
  \end{equation*}
  We have already shown that the coefficients $a_n$ are positive and fixed by
  $\mathcal{L}$. Hence they are log-concave and, as a consequence, unimodal.
  Because the sequence $(a_n)$ is symmetric and unimodal, its maximum is
  $a_N$, with $N = \lfloor (s - 3) / 2 \rfloor$, and it follows that $(a_n)$
  is $r$-factor log-concave if and only if $r \leq \frac{a_N}{a_N - 1}$.
  
  The value $a_N$ can be obtained from the expansion \eqref{eq:ps1sin}.
  Indeed, writing $\zeta_s = e^{2 \pi i / s}$, we find that, for $n = 0, 1,
  \ldots, s - 3$,
  \begin{equation*}
    a_n - 1 = \frac{1}{\sin (2 \pi / s)} \Im \sum_{j = 1}^n \zeta_s^{j
     + 1} = \frac{1}{\sin (2 \pi / s)} \Im \left[ \zeta_s^2  \frac{1 -
     \zeta_s^n}{1 - \zeta_s} \right] .
  \end{equation*}
  A simple calculation shows that
  \begin{eqnarray*}
    \Im \left[ \zeta_s^2  \frac{1 - \zeta_s^n}{1 - \zeta_s} \right] & =
    & \Im \left[ \frac{- \zeta_s + \zeta_s^2 + \zeta_s^{n + 1} -
    \zeta_s^{n + 2}}{|1 - \zeta_s |^2} \right]\\
    & = & \frac{- \sin ( \frac{2 \pi}{s}) + \sin ( \frac{4 \pi}{s}) + \sin (
    \frac{2 \pi (n + 1)}{s}) - \sin ( \frac{2 \pi (n + 2)}{s})}{2 - 2 \cos (2
    \pi / s)} .
  \end{eqnarray*}
  In the case when $n$ is $N = \lfloor (s - 3) / 2 \rfloor$, the sines can all
  be expressed with arguments in terms of $s$ only, and one obtains
  \begin{align*}
    a_N - 1 = \frac{1}{\sin (2 \pi / s)} \begin{cases}
      \frac{\sin (4 \pi / s)}{2 - 2 \cos (2 \pi / s)}, & \text{if $s$ is
      even,}\\
      \frac{2 \sin (\pi / s) - \sin (2 \pi / s) + \sin (4 \pi / s)}{2 - 2
      \cos (2 \pi / s)}, & \text{if $s$ is odd.}
    \end{cases}
  \end{align*}
  Using $\frac{a_N}{a_N - 1} = 1 + \frac{1}{a_N - 1}$, it is now
  straightforward to verify the claim using standard trigonometric identities.
  Indeed, since $\cos(2t)=2\cos(t)^2 - 1$, we have, by definition of the Chebyshev polynomials $U_n$,
  \begin{align*}
    a_N - 1 = \frac{1}{4 (1-t^2) U_1(t)} \begin{cases}
      U_3(t), & \text{if $s$ is even,}\\
      2U_0(t) - U_1(t) + U_3(t), & \text{if $s$ is odd,}
    \end{cases}
  \end{align*}
  with $t=\cos(\pi/s)$, from which the claim is immediate.
\end{proof}

\begin{example}
  \label{eg:ps1}The first few polynomials of Theorem~\ref{thm:ps1} are $p_3
  (x) = 1$, $p_4 (x) = 1 + x$, as well as
  \begin{eqnarray*}
    p_5 (x) & = & 1 + \tfrac{1 + \sqrt{5}}{2} x + x^2,\\
    p_6 (x) & = & 1 + 2 x + 2 x^2 + x^3,\\
    p_7 (x) & = & 1 + (1 + \alpha) x + \alpha (1 + \alpha) x^2 + (1 + \alpha)
    x^3 + x^4, \quad \alpha = 2 \sin \left( \tfrac{3 \pi}{14}
    \right),\\
    p_8 (x) & = & 1 + (1 + \sqrt{2}) x + (2 + \sqrt{2}) x^2 + (2 + \sqrt{2})
    x^3 + (1 + \sqrt{2}) x^4 + x^5 .
  \end{eqnarray*}
  In each case, positivity of the coefficients, together with the fact that
  they are fixed under $\mathcal{L}$, shows that they are
  $\infty$-log-concave. None of these polynomials has all their roots on the
  negative real axis. Hence, it is not possible to deduce the
  $\infty$-log-concavity of their coefficients using Br\"and\'en's result in
  form of Corollary~\ref{cor:realrootslc}.
\end{example}

\begin{remark}
  \label{rk:rtoolow}Note that $1 / \cos (2 \pi / s) \rightarrow 1$ as $s
  \rightarrow \infty$. Part \ref{l:ps:r} of Theorem~\ref{thm:ps1} therefore
  shows that, given any $\varepsilon > 0$, there is a finite sequence $(a_n)$
  which is $\infty$-log-concave but not $(1 + \varepsilon)$-factor
  log-concave.
  
  In particular, let $(a_n)$ be the coefficients of the polynomial $p_s (x)$
  with $s \geq 6$. The above shows that $(a_n)$ is fixed by $\mathcal{L}$
  but is not $r$-factor log-concave for $r > 2$. In particular, there is no $m
  > 0$ such that $\mathcal{L}^m (a_n)$ is $r$-factor log-concave for $r > 2$.
  It is therefore not possible to apply Lemma~\ref{lem:rfactor} as in
  Example~\ref{eg:rfactorinfty} to show that the sequence $ (a_n)$ is
  $\infty$-log-concave.
\end{remark}

Apart from the special case discussed in Remark~\ref{rk:rtoolow}, we have been
able to successfully apply the approach of Example~\ref{eg:rfactorinfty} to
establish $\infty$-log-concavity in all the examples we have encountered. This
motivates the next question which, in particular, asks whether
$\infty$-log-concavity of a finite sequence is decidable.

\begin{question}
  \label{q:lc}Given a finite positive sequence $(a_n)$, compute $\mathcal{L}^m
  (a_n)$ for $m = 1, 2, \ldots$ until either
  \begin{enumeratealpha}
    \item \label{l:neg}$\mathcal{L}^m (a_n)$ has negative terms, or
    
    \item \label{l:inff}$\mathcal{L}^m (a_n) = \lambda (a_n)$ for some
    $\lambda > 0$, or
    
    \item \label{l:infr}$\mathcal{L}^m (a_n)$ is $r$-factor log-concave for
    some $r \geq \frac{3 + \sqrt{5}}{2} \approx 2.618$.
  \end{enumeratealpha}
  In the case \ref{l:neg} the sequence is $(m - 1)$-log-concave but not
  $m$-log-concave, and in the cases \ref{l:inff} and \ref{l:infr} the sequence
  is $\infty$-log-concave.
  
  Does this simple algorithm always terminate? Or, if this is not the case, is
  there some other algorithm which determines, in finite time, whether a given
  finite sequence is $\infty$-log-concave?
\end{question}

\section{Sequences fixed by $\mathcal{L}^2$}\label{sec:fixedbyL2}

We now consider an analog of Theorem~\ref{thm:fixL} which characterizes
sequences fixed by $\mathcal{L}^2$. No such characterization appears to exist
for sequences fixed by $\mathcal{L}^n$ with $n > 2$; see
Example~\ref{eg:fixL3}. Note that the rational generating function in
Theorem~\ref{thm:fixL2} is equivalent to the sequence $(a_n)$ satisfying the
recurrence
\begin{equation*}
  a_n - \beta a_{n - 1} + (\beta^2 - \gamma) a_{n - 2} - a_{n - 3} = 0,
\end{equation*}
with initial conditions $a_0 = 1$ and $a_{- 1} = a_{- 2} = a_{- 3} = 0$.

\begin{theorem}
  \label{thm:fixL2}Let $\beta, \gamma$ be arbitrary numbers. Then the sequence
  $(a_n)_{n \geq 0}$ defined by
  \begin{equation*}
    \sum_{n \geq 0} a_n x^n = \frac{1}{1 - \beta x + (\beta^2 - \gamma)
     x^2 - x^3}
  \end{equation*}
  is fixed by $\mathcal{L}^2$. Note that $a_0 = 1$, $a_1 = \beta$, $a_2 =
  \gamma$.
\end{theorem}

\begin{proof}
  Let $\alpha_1, \alpha_2, \alpha_3$ be such that
  \begin{equation*}
    1 - \beta x + (\beta^2 - \gamma) x^2 - x^3 = (\alpha_1 - x) (\alpha_2 -
     x) (\alpha_3 - x) .
  \end{equation*}
  Necessarily, $\alpha_1 \alpha_2 \alpha_3 = 1$. In the sequel, we assume that
  $\alpha_1, \alpha_2, \alpha_3$ are distinct; the general case may then be
  obtained by a limiting argument. Using the general partial fraction
  expansion
  \begin{equation*}
    \prod_{j = 1}^d \frac{1}{\alpha_j - x} = \sum_{k = 1}^d \frac{1}{\alpha_k
     - x} \prod_{\substack{
       j = 1\\
       j \neq k
     }}^d \frac{1}{\alpha_j - \alpha_k},
  \end{equation*}
  valid for distinct $\alpha_j$, we obtain the expansion
  \begin{equation*}
    \prod_{j = 1}^d \frac{1}{\alpha_j - x} = \sum_{n = 0}^{\infty} x^n
     \sum_{k = 1}^d \frac{1}{\alpha_k^{n + 1}}
     \prod_{\substack{
       j = 1\\
       j \neq k
     }}^d \frac{1}{\alpha_j - \alpha_k} .
  \end{equation*}
  In the case $d = 3$, we thus have
  \begin{equation}
    a_n = \frac{1}{(\alpha_1 - \alpha_2) (\alpha_2 - \alpha_3) (\alpha_3 -
    \alpha_1)} \left[ \frac{\alpha_3 - \alpha_2}{\alpha_1^{n + 1}} +
    \frac{\alpha_1 - \alpha_3}{\alpha_2^{n + 1}} + \frac{\alpha_2 -
    \alpha_1}{\alpha_3^{n + 1}} \right] . \label{eq:L2:an}
  \end{equation}
  By a straight-forward direct computation, using $\alpha_1 \alpha_2 \alpha_3
  = 1$, we find that
  \begin{equation*}
    a_n^2 - a_{n - 1} a_{n + 1} = \frac{1}{(\alpha_1 - \alpha_2) (\alpha_2 -
     \alpha_3) (\alpha_3 - \alpha_1)} \left[ \frac{\alpha_3 -
     \alpha_2}{\alpha_1^{- n - 2}} + \frac{\alpha_1 - \alpha_3}{\alpha_2^{- n
     - 2}} + \frac{\alpha_2 - \alpha_1}{\alpha_3^{- n - 2}} \right] .
  \end{equation*}
  Comparing this expression for $\mathcal{L} (a_n)$ with \eqref{eq:L2:an} for
  $(a_n)$, it becomes clear that iterating this computation to compute
  $\mathcal{L}^2 (a_n)$ will result in the original sequence. In other words,
  we have shown that $\mathcal{L}^2 (a_n) = (a_n)$, as claimed.
  
  An alternative and automatic, but arguably less illuminating and rather more
  computational, proof can be obtained by applying the $C$-finite ansatz
  \cite{zeilberger-cfinite} as indicated in Remark~\ref{rk:cfinite}.
\end{proof}

\begin{example}
  If $\beta = 2$ and $\gamma = 3$, then we obtain the sequence $(a_n)$ whose
  first few terms are given by
  \begin{equation*}
    1, 2, 3, 5, 9, 16, 28, 49, 86, 151, 265, \ldots
  \end{equation*}
  Indeed, this sequence \cite[A005314]{oeis} may be written as
  \begin{equation*}
    a_{n - 1} = \sum_{k = 0}^{\lfloor (n - 1) / 3 \rfloor} \binom{n - k}{2 k
     + 1},
  \end{equation*}
  and $a_{n - 1}$ counts the number of compositions of $n$ into parts
  congruent to $1$ or $2$ modulo $4$. Theorem~\ref{thm:fixL2} shows that this
  sequence is fixed by $\mathcal{L}^2$. However, $\mathcal{L} (a_n) = (1, 1, -
  1, - 2, 1, 4, \ldots)$ is not positive, so $(a_n)$ is not even log-concave.
\end{example}

\begin{example}
  \label{eg:fixL3}After the success of Theorems~\ref{thm:fixL} and
  \ref{thm:fixL2}, one may be tempted to hope that similar characterizations
  exist for fixed sequences of $\mathcal{L}^m$ when $m > 2$. This does not
  seem to be the case, as is illustrated by the following example for $m = 3$.
  The sequence $(a_n)$, characterized by $a_0 = 1$, $a_1 = 2$, $a_2 = 5$, $a_3
  = 9$ together with the fact that it is fixed by $\mathcal{L}^3$, is given by
  \begin{equation*}
    1, 2, 5, 9, \frac{96}{5}, \frac{324547}{9450}, \frac{4079971657981}{58296
     672000}, \ldots
  \end{equation*}
  Computing several more terms, one can check that $(a_n)$ cannot satisfy a
  linear recurrence of small degree and small order. For instance, we have found that
  it does not satisfy a recurrence with constant coefficients of order up to
  $10$ (recall that in the cases $m = 1$ and $m = 2$ any such sequence
  satisfies a recurrence with constant coefficients of order $3$). We also
  found that $(a_n)$ does not satisfy a recurrence with linear coefficients of
  order up to $6$, or a recurrence with quadratic coefficients of order up to
  $4$.
\end{example}

In all the examples we have considered of sequences that are fixed by
$\mathcal{L}^3$ with initial integral values as in Example~\ref{eg:fixL3}, the
resulting sequences involved fractions of rapidly increasing size. We are
therefore lead to wonder whether there are examples of integer sequences fixed
by $\mathcal{L}^3$ (but not by $\mathcal{L}$).

\begin{question}
  For $n > 2$, are there (positive) integer sequences that are fixed by
  $\mathcal{L}^n$, but not by $\mathcal{L}^m$ for any $1 \leq m < n$?
\end{question}

\section{Convolutions of sequences}\label{sec:powers}

It is a well-known result of Hoggar \cite{hoggar-lc} that the product of
polynomials with positive and log-concave coefficients again has log-concave
coefficients. This is generalized in \cite{johnson-lc} by Johnson and
Goldschmidt who apply their result to show, for instance, log-concavity of the
Stirling numbers of the second kind as a sequence in the second parameter.

This naturally leads one to wonder whether anything interesting can be said
about the log-concavity of the coefficients of the product $p (x) \cdot q (x)$
if $p (x)$ and $q (x)$ have $m$- and $n$-log-concave coefficients,
respectively. The next example, constructed as an application of the previous
sections, dampens excessive expectations by demonstrating that the product of
two polynomials with $\infty$-log-concave coefficients may fail to be
$5$-log-concave.

\begin{example}
  \label{eg:fail}Consider the polynomial
  \begin{equation}
    p (x) = 1 + (1 + \sqrt{2}) x + (2 + \sqrt{2}) x^2 + (2 + \sqrt{2}) x^3 +
    (1 + \sqrt{2}) x^4 + x^5 \label{eq:p81}
  \end{equation}
  from Example~\ref{eg:ps1}. Since it is fixed under $\mathcal{L}$, its
  coefficients are $\infty$-log-concave. However, the coefficients of $p
  (x)^2$ are $4$-log-concave but not $5$-log-concave. On the other hand, the
  coefficients of $p (x)^n$ for $n = 3, 4, \ldots, 50$ are again
  $\infty$-log-concave.
  
  We remark that slight perturbations of the polynomial $p (x)$ in
  \eqref{eq:p81} can be used to construct many further examples of polynomials
  with similar properties but which are not fixed under $\mathcal{L}$. For
  instance, the degree $6$ polynomial $q (x) = p (x) + (x / 4)^6$ has the
  property that its coefficients are $5$-log-concave, while the coefficients
  of $q (x)^2$ are only $4$-log-concave but not $5$-log-concave.
\end{example}

Note that Example~\ref{eg:fail} demonstrates that, in general, if the
coefficients of $p (x)^{\lambda}$ are $m$-log-concave for some $m > 1$ this
does not imply that the coefficients of $p (x)^{\lambda + 1}$ are
$m$-log-concave as well. On the other hand, note that each polynomial $p (x) =
a_0 + a_1 x + \cdots + a_d x^d$ with positive coefficients $a_j > 0$, $j = 0,
1, \ldots, d$, induces a probability distribution on the set $\{0, 1, \ldots,
d\}$ with probability weights specified by the coefficients of $p (x)$.
Namely, a random variable $X$ is distributed according to $p (x)$, if, for $j
= 0, 1, \ldots, d$,
\begin{equation*}
  \op{Prob} (X = j) = \frac{a_j}{p (1)} .
\end{equation*}
If $X$ and $Y$ are random variables distributed according to $p (x)$ and $q
(x)$, respectively, then $X + Y$ is distributed according to the product $p
(x) q (x)$. In particular, let $X_1, \ldots, X_{\lambda}$ be independent
random variables distributed according to $p (x)$. Then their sum $X_1 +
\cdots + X_{\lambda}$ is distributed according to $p (x)^{\lambda}$. The
central limit theorem therefore suggests that the coefficients of $p
(x)^{\lambda}$ should eventually become ``more log-concave'' as $\lambda$
increases. Indeed, in all the examples we have considered, we have observed
that the coefficients of $p (x)^{\lambda}$ become $\infty$-log-concave for
finite $\lambda$.

\begin{conjecture}
  \label{conj:powers}Let $p (x)$ be a polynomial with positive coefficients,
  that is $p (x) = a_0 + a_1 x + \cdots + a_d x^d$ with $a_j > 0$ for all $j =
  0, 1 \ldots, d$. Then there exists $N$ such that, for all $\lambda \geq
  N$, $p (x)^{\lambda}$ has coefficients that are $\infty$-log-concave.
\end{conjecture}

\begin{example}
  Consider $p (x) = 1 + x + x^2$. Computer experiments suggest that the
  coefficients of $p^{\lambda}$ are $\infty$-log-concave for $\lambda
  \geq 10$. For $10 \leq \lambda \leq 500$ we have proved this
  using Lemma~\ref{lem:rfactor} as in Example~\ref{eg:rfactorinfty}. For
  instance, in the case $\lambda = 10$, one observes that the coefficients of
  $\mathcal{L}^5 [p^{10}]$ are $9.10$-factor log-concave. More generally, it
  appears that $p^{\lambda}$ is log-concave for $\lambda \geq 1$ (and strongly
  log-concave for $\lambda \geq 2$), $2$-log-concave for $\lambda \geq
  4$, $3$-log-concave for $\lambda \geq 7$, $4$-log-concave for $\lambda \geq
  8$, $5$-log-concave for $\lambda \geq 9$, and, as mentioned above,
  $\infty$-log-concave for $\lambda \geq 10$. Note that this information is
  included in Table~\ref{tbl:exponents}.
\end{example}

\begin{example}
  Consider the polynomial $p (x) = 1 + x + 2 x^2$ and note that the
  coefficients $(1, 1, 2)$ are not log-concave. However, $p^3$ has log-concave
  coefficients, $p^{10}$ has $2$-log-concave and $p^{16}$ has $3$-log-concave
  coefficients. This information is further extended in
  Table~\ref{tbl:exponents}. In particular, it turns out that $p^{23}$ has
  $\infty$-log-concave coefficients. Indeed, the coefficients of
  $\mathcal{L}^5 [p^{23}]$ are $4.23$-factor log-concave.
\end{example}

\begin{example}
  Let us consider general quadratic polynomials $p (x) = a_0 + a_1 x + a_2
  x^2$. Note that scaling a polynomial $p (x) \rightarrow \lambda p (x)$ does
  not affect questions of log-concavity, and neither does the transformation
  $p (x) \rightarrow p (\lambda x)$. Without loss of generality, we may
  therefore assume $a_0 = 1$ and $a_1 = 1$. In the case of the polynomials $p
  (x) = 1 + x + a x^2$, $a \in \{1, 2, \ldots, 6\}$, and values $m \in \{1, 2,
  \ldots, 10, \infty\}$, Table~\ref{tbl:exponents} lists the minimal exponent
  $\lambda$ such that $p^{\lambda}$ has $m$-log-concave coefficients.
  
  \begin{table}[h]
    \begin{tabular}{|c|c|c|c|c|c|c|c|c|c|c|c|}
      \hline
      & $1$ & $2$ & $3$ & $4$ & $5$ & $6$ & $7$ & $8$ & $9$ & $10$ &
      $\infty$\\
      \hline
      $1 + x + x^2$ & $1$ & $4$ & $7$ & $8$ & $9$ & $10$ & $10$ & $10$ & $10$
      & $10$ & $10$\\
      \hline
      $1 + x + 2 x^2$ & $3$ & $10$ & $16$ & $20$ & $21$ & $22$ & $23$ & $23$ &
      $23$ & $23$ & $23$\\
      \hline
      $1 + x + 3 x^2$ & $5$ & $16$ & $26$ & $31$ & $33$ & $35$ & $36$ & $36$ &
      $36$ & $36$ & $36$\\
      \hline
      $1 + x + 4 x^2$ & $7$ & $22$ & $35$ & $42$ & $46$ & $48$ & $49$ & $49$ &
      $49$ & $49$ & $49$\\
      \hline
      $1 + x + 5 x^2$ & $9$ & $28$ & $45$ & $53$ & $58$ & $61$ & $61$ & $62$ &
      $62$ & $62$ & $62$\\
      \hline
      $1 + x + 6 x^2$ & $11$ & $34$ & $54$ & $65$ & $70$ & $73$ & $74$ & $75$
      & $75$ & $75$ & $75$\\
      \hline
    \end{tabular}
    \caption{\label{tbl:exponents}Minimal exponent $\lambda$ such that
    $p^{\lambda}$ has $m$-log-concave coefficients.}
  \end{table}
  
  We invite the reader to observe the various loose patterns suggested by the
  data contained in Table~\ref{tbl:exponents}.
\end{example}

\begin{acknowledgements}
We wish to thank Danylo Radchenko for interesting
and helpful comments. The second author would like to thank the
Max-Planck-Institute for Mathematics in Bonn, where part of this work was
completed, for providing wonderful working conditions.
\end{acknowledgements}

\end{document}